\documentclass{amsart}
\usepackage{amscd,amssymb,amsmath,amsthm}
\usepackage[dvips]{graphicx}
\usepackage[all]{xy}
\theoremstyle{plain}
\newtheorem{definition}{Definition}
\newtheorem{proposition}{Proposition}
\newtheorem{theorem}[proposition]{Theorem}

\newtheorem{lemma}[proposition]{Lemma}

\newtheorem*{proposition*}{Proposition}
\newtheorem*{theorem*}{Theorem}
\newtheorem*{corollary*}{Corollary}
\newtheorem*{lemma*}{Lemma}
\newtheorem*{remark*}{Remark}
\newtheorem*{example*}{Example}

\newcommand{\Z}{\mathbb{Z}}
\newcommand{\Q}{\mathbb{Q}}
\newcommand{\R}{\mathbb{R}}
\newcommand{\C}{\mathbb{C}}

\begin{document}

\title{Kontsevich-Soibelman Wall Crossing Formula and Holomorphic Disks}

\author{Vito Iacovino}


\email{vito.iacovino@gmail.com}

\date{version: \today}


\begin{abstract} We define rational numbers associated to moduli space of holomorphic disks bounding a complex lagrangian submanifold on a hyperkhaler manifold of real dimension four. We provide a simple a direct proof of Kontsevich-Soibelman Wall Crossing Formula for these rational invariants.


\end{abstract}

\maketitle

\section{Introduction}

In \cite{KS} Kontsevich and Soibelman proposed a formula (KSWCF) on how Donaldson-Thomas invariants jump when we cross a wall of marginal stability (also called wall of the first type). Soon after, this formula as been stundied intesively by physicists in the context of $N=2$ supersymmetric gauge theories, in particular by Gaiotto, Moore and Neitze (\cite{GMN1},\cite{GMN2},\cite{GMN3}). In \cite{CV}, Cecotti and Vafa proposed an alternative more direct argument derivation the Kontsevich-Soibelman Wall Crossing Formula for all the $N = 2$ theories that can be associated to a Seiberg-Witten curve. It is shown that the computation of BPS degeneracy can be mapped to a computation of open topological $A$-model strings and that the KSWCF follows from the invariance of the open topological partition function when we cross a wall of marginal stability. 

In \cite{OGW3} we provide the mathematical definition of the Open Topological $A $-model String partition function. The result of \cite{CV} suggest that there should be an interpretation of KSWCF in the setting of \cite{OGW3}. In this paper we prove that this is indeed the case. We first define rational numbers associated to moduli spaces of holomorphic disks with boundary mapped on the Seiberg-Witten curve and then we prove that the KSWCF is a simple consequence of the  invariance in Open Gromov-Witten. 

More in general, we consider $(X,\Omega)$ complex symplectic manifold of real dimension four, and $\Sigma$ a complex-lagrangian submanifold of $X$.  The pair $(X, \Sigma)$ is not necessary compact. For non-compact geometries we assume suitable convexity property at infinity which ensure compactness of the moduli space of curves. An important example relevant in physics is the Seiberg-Witten curve for theories of class $\mathcal{S}$ studied in \cite{GMN2}, where $X=T^*C$ for a $C$ a  Riemannian surface, and $\Sigma \subset T^*C$ is a branch cover of $C$.

We introduce Multi Disk Homology with central charges that is a variant of Multi Curve Homology introduced in \cite{OGW3}. The moduli space of multi-disks are built from the moduli space of pseudolomorphic disks bounding $\Sigma$, with respect to a family of  almost complex structure compatible with the symplectic structure $\text{Re} (e^{- i \theta} \Omega)$, as we vary $\theta \in \R / \Z$. 
The same consideration as in \cite{OGW3} allows us to associate to the moduli space of multi disks an element $W \in MDH_0$, that is a well defined invariant associated to the pair $(X,\Sigma)$ as long as we do not hit a wall of second type in the sense of \cite{KS}. 

Multi-disks can be considered as the analogous in topological string of the multi BPS particles of supersymmetric gauge theories, that play a crucial role on the construction of hyperkhaler metric of \cite{GMN1}. 
This analogy motivated the introduction of the term "multi-disk" in \cite{OGW1}.

To define rational numbers we use the multilinking homomorphism $MDH_0 \rightarrow H_0(pt)$ introduced in \cite{OGW3}. In our case the definition depends on the choice of central charges. 
Outside a wall of the first type the central charges define rational numbers invariants, locally constant on the space of central charges, but that can jump when we cross a wall of first type. The Kontsevich-Soibelman Wall Crossing Formula follows from the invariance of $W$ as element in $MDH_0$.


The sign of rational numbers defined in \cite{OGW3} in the compact case depend on the choice of a spin structure of the lagrangian submanifold, unless $[\partial \beta] =0$ in $H_1(L,\Z)$ . In the context of this paper we define rational numbers that do not depend on the spin structure.
A classical result of Atiyah claims that Spin structures on the surface $\Sigma$ are in correspondence to quadratic refinement of the intersection form of $H_1(\Sigma, \Z)$. We modify the rational  invariants using this correspondence to get invariants not depending on  the choice of the spin structure. This modification is responsible for the appearance of the factor $ (-1)^{ \langle \bullet , \bullet \rangle } $ in the Kontsevich-Soibelman Lie Algebra.

Cecotti and Vafa in \cite{CV}  associate to a two dimensional pair $(T^* \C^2,\Sigma)$  a three dimensional  pair $(T^*\R^3,L)$ and they claim that the number of BPS invariants of $(T^*\C^2,\Sigma)$  are computable in terms of topological string of $(T^*\R^3,L)$. However their argument  suffer of a big gap due to the fact that the manifold  $L$ they consider is not lagrangian (it is lagrangian at the conformal limit, but not in the smooth case). In \cite{CCV} this problem is fixed defining an honest lagrangian submanifold $L \subset T^*\R^3$ using the so called $R$-flow. 
In this paper we instead  use the augmented moduli space of pseudoholomorphic curves with target $(X, \Sigma )$. The moduli space of disks has augmented dimension zero, but the moduli space of  higher genus curves have dimension greater than zero, making hard to define  rational invariants. For this reason we restrict  to the space of multi-disks that correspond to the  semiclassical limit  of the KSWCF, in the sense of \cite{KS}.  In \cite{refined} we show how the quantum algebra of \cite{KS} is related to the geometric structures of \cite{OGW3} associated to higher genus open holomorphic curves.  

In \cite{floer} we apply the method of this paper to define Holomorphic Lagrangians Floer Homology for a pair of holomorphic Lagrangians. In a particular case this gives the $2d-4d$ Wall Crossing Formula of \cite{GMN4}. The existence of such homology is a recent proposal of Kontsevich.

\section{Multi Disk Homology with Central Charges}

Assume it is given:
\begin{itemize}
\item an oriented riemmanian surface (non necessarily compact) $\Sigma$,
\item 
 a free abelian group $\Gamma$  called \emph{topological charges},
\item an homorphism of abelian groups
$$ \partial : \Gamma \rightarrow H_1(\Sigma, \Z) .$$

\end{itemize}
In this section we define Multi Disk Homology with central charges as a variant of Multi-Curve Homology defined in \cite{OGW3}. We denote it by $MDH$. 
It is associated to the following data:
\begin{itemize}
\item an homomorphism of abelian groups 
$$ Z : \Gamma \rightarrow \C $$
called \emph{central charge},
\item a strictly convex sector of the complex plane $S \subset \C$,
\item  
A quadratic form on $\Gamma_{\R} = \Gamma \otimes_{\Z} \R$ such that 
$$  Q|_{\text{Ker}(Z)} < 0 . $$
\end{itemize}
Define 
$$ \Gamma(S,Z, Q) = \text{non-negative integral combination of }  \{ \beta \in \Gamma | Z(\beta) \in S, Q(\beta) \geq 0 \} .$$


\begin{definition}
A decorated connected tree $T$ with topological charge $\Gamma$ is defined by the following data
\begin{itemize}
\item A finite set $H(T)$ of half-edges of $T$.
\item A finite set $V(T)$ of vertices of $T$. 
\item An involution $ \sigma : H(T) \rightarrow H(T)$ without fixed points.  
The set of orbits of $\sigma$ is denoted $E(T)$, and is called the set of edges of $G$.
\item A map $\pi: H(T) \rightarrow V(T)$, which sends a half-edge to the vertex to which it is attached, 
\item For each $v \in V(T)$ a topological charge $\beta_v \in \Gamma$.
\end{itemize}
We require that 
\begin{itemize}
\item all the elements of $V(T)$ are equivalent by the relation generated by $\pi(h) \sim \pi(\sigma(h))$ for $h \in H(T)$,  
\item $$ |V(T)| = |E(T)|  +1 .$$
\end{itemize}
The homology class of $T$ is defined by $\beta = \sum_{v \in V(T)} \beta_v$.

A vertex $v$ is called unstable if $\beta_v = 0$ and $| \pi^{-1}(v) | \leq 2 $.
The tree $T$ is called stable if it has not unstable vertices.

A unconnected tree is a union of connect trees in the obvious sense. 
\end{definition}



Give $e= \{ h_1,h_2  \} \in E(T)$ we define a new tree $ \delta_e T $ as follows: 
\begin{itemize}
\item $H(\delta_e T) = H(T) \setminus \{ h_1, h_2  \}$. 
\item $V(\delta_e T) = V(T) / \sim $, where we identify $ \pi(h_1) \sim \pi(h_2) $. Let $v_0 \in V(\delta_e T)$  the image of $ \pi(h_1) \sim \pi(h_2) $.  
\item for $v \in V(\delta_e T) \setminus \{v \}$, $\beta_v'= \beta_v$,
\item $\beta_{v_0}' = \beta_{\pi(h_1)} + \beta_{\pi(h_2)} $.
\end{itemize}
Here we denoted with $\beta'$ the topological charges of $\delta_e (T)$ in order to distinguish them by the topological charges $\beta$ of $T$.

Let $ \mathfrak{T}_{k}(\beta, S, Z,Q)$ be the set of (non necessarily connected) trees with $k$ edges of topological charge $\beta$ and $\beta_v \in \Gamma(S,Q)$ for every $v \in V(T)$. Observe that $ \mathfrak{T}_{k}(\beta, S,Z, Q)$ has finite cardinality since any $\beta \in \Gamma$ can be represented as a sum of other elements of $\Gamma(S,Z,Q)$ in finitely many ways.

Denote by $ \mathfrak{T}_{k,l}(\beta, S,Z, Q)$ the set of pairs $(T, \{ E_i  \}_{0 \leq i \leq l})$ where $T \in  \mathfrak{T}_{k}(\beta, S,Z, Q)$  and $ E_0 \subset E_1 \subset ... \subset E_l$ is an increasing sequence of subsets of $E(T)$. 

We consider two operations on $ \mathfrak{T}(\beta, S,Z,Q)$. Let $(T, \{ E_0, E_1,..., E_l \}) \in \mathfrak{T}_{k,l}(\beta,S,Z,Q)$  
\begin{itemize}
\item For $0 \leq i \leq l$, $\partial_i (T, \{ E_0, E_1,..., E_l \})$ is defined by $(T, \partial_i  \{ E_0, E_1,..., E_l \}) \in \mathfrak{T}_{k,l-1}(\beta,S,Z,Q)$, where $\partial_i  \{ E_0, E_1,..., E_l \} =  \{ E_0,...,\hat{E_i},..., E_l \}$.
\item For $e \in E(T) \setminus E_l$, define $\delta_e(T, \{ E_0, E_1,..., E_l \}) = (\delta_e T, \{ E_0, E_1,..., E_l \})  \in \mathfrak{T}_{k-1,l}(\beta,S,Z,Q)$.
\end{itemize}


Given a strictly convex sector $S$, define 
$$ I_S = \{l  \text{ ray on } \R^2| \text{ } l \subset S   \} .$$
We consider $I_S$ as a one dimensional manifold diffeomeomorphic to an interval in the obvious sense.

For $(T,m) \in \mathfrak{T}_{k,l}(\beta,S,Z,Q)$,
let $\mathcal{C}_d(T,m)$ be the space of currents on $(\Sigma \times I_S)^{H(T)}$ of dimension $d+2k+l$ that can be represented by a smooth simplicial chain with twisted coefficients on $ \mathfrak{o}_{T}$, transverse to the big diagonal associated to any edge in $E(T) \setminus E_l$ and invariant under the group $\text{Aut}(T,m)$.
Here, $ \mathfrak{o}_T = ( \otimes_{e \in E(T)} \mathfrak{o}_e )  $, 
where $ \mathfrak{o}_e = \{ \text{set of orientations of the edge $e$} \} \cong \Z_2 $.

Let $\mathcal{C}_{d,k,l}(\beta,S,Z,Q) = \oplus_{(T,m) \in \mathfrak{T}_{k,l}(\beta,S,Z,Q)}  \mathcal{C}_d(T,m)$. 
 As in  \cite{OGW3}, we can define on $\mathcal{C}_{d,k,l}(\beta,S,Q)$ three operations
$$\partial: \mathcal{C}_{d,k,l}(\beta,S,Z,Q) \rightarrow\mathcal{C}_{d-1,k,l}(\beta,S,Z,Q)$$
$$\delta: \ \mathcal{C}_{d,k,l}(\beta,S,Z,Q)\rightarrow \mathcal{C}_{d-1,k-1,l}(\beta,S,Z,Q)$$
$$\tilde{\partial}: \mathcal{C}_{d,k,l}(\beta,S,Z,Q) \rightarrow \mathcal{C}_{d-1,k,l+1}(\beta, S,Z,Q) $$

such that the linear map 
\begin{equation} \label{boundary}
\hat{\partial} = \partial - \delta - \tilde{\partial} : \mathcal{C}_{d}(\beta,S,Z,Q)  \rightarrow \mathcal{C}_{d-1}(\beta,S,Z,Q)
\end{equation}
square to zero, where  $  \mathcal{C}_{d}(\beta,S,Z,Q) = \oplus_{k,l} \mathcal{C}_{d,k,l}(\beta,S,Z,Q)$. 

To define multi disk homology we need to consider a subcomplex of (\ref{boundary}) as follows.
\begin{definition} \label{forget}
An element $B \in  \mathcal{C}_{d}(\beta,S,Z,Q)$ is called forgetful compatible if the following happen. For $(T, \{E_0,E_1,...,E_l \}) \in \mathfrak{T}_{k,l}(\beta;S,Z,Q)$ define  $(T', \{E_0', E_1',..., E_l' \}) \in \mathfrak{T}_{k-k_0,l}(\beta,S,Z,Q)$, where $T'$ is defined by the data
$V(T' ) = V(T)$, $E(T') = E(T) \setminus E_0$, $\beta_v'= \beta_v$ for $v \in V(G)$, and  $E_i '= E_i \setminus E_0$ for $0 \leq i \leq l$.

If $T'$ is unstable then $B(T, \{E_0, E_1,..., E_l \}) =0$.
If $T'$ is stable, there exists
\begin{itemize}
\item a representation as simplicial chain
$$B(T', \{E_0', E_1',..., E_l' \}) = \sum_{ a \in  \mathcal{A}} \rho_a  \phi_a'   $$
with $\phi'_a:\Delta_{l+d+k} \rightarrow L^{H(T')}$, $\rho_a \in \Q$,
\item $\forall a \in \mathcal{A}, \forall v \in V(T)$,  a one dimension compact manifold $F_{a,v},$
\item   $\forall a \in \mathcal{A}, \forall v \in V(T)$ a map $f_{a,v} : X_a \times F_{a,v} \rightarrow L$ with 
$$ (f_{a,v})_* [\{ x \} \times F_{a,v}] =    \partial \beta_{v}  \text{ in } H_1(L, \Z) ,$$
\end{itemize}
such that 
$$B(T, \{E_0, E_1,..., E_l \}) = \sum_{ a \in  \mathcal{A}} \rho_a \phi_a$$
for 
$$  \phi_a (x, \{ y_h \}_{h \in H_0} ) = (  \phi_a' (x) , \{ f_{a,\pi(h)}(x,y_h) \}_{h \in H_0} ) .$$ 
We require that 
the $1$-currents $(f_{a,v})_* [\{ x \} \times F_{a,v)}]$  depends only on $(T', \{E_0', E_1',..., E_l' \})$, $v \in V(T')$ and $x \in X_a$,  and does not depend on $(T, \{E_0,E_1,...,E_l \}) $.
\end{definition}

Let $\hat{\mathcal{C}}_{d}(\beta,S,Z,Q)$ the subset of $\mathcal{C}_{d}(\beta,S,Z,Q)$ given the multi-curve chains that are  forgetful compatible in the sense of Definition \ref{forget}. It is easy to see $\hat{\mathcal{C}}_{d}(\beta,S,Z,Q)$ is a subcomplex, so that (\ref{boundary}) descends to a derivation 
\begin{equation} \label{boundary0}
\hat{\partial}_d  : \hat{\mathcal{C}}_{d}(\beta,S,Z,Q)  \rightarrow \hat{\mathcal{C}}_{d-1}(\beta,S,Z,Q).
\end{equation}
Multi curve homology is homology of this complex
$$ MCH_d(\beta,S,Z,Q) =  \frac{\text{Ker}( \hat{\partial}_d) }{\text{Im}( \hat{\partial}_{d+1}) } .   $$

\subsection{Nice Multi-Disk Homology}
In analogy with \cite{OGW3}, we now define the space of \emph{nice multi-disk chains}. 
\begin{definition}
A multi curve chain $B$ is called $\emph{nice}$ if $B_{d,k,l}=0$ for each $l>0$. 
\end{definition}
From the definition it follows that $B(T) := B(T,  \{E_0 \} )$ does not depends on $E_0 \subset E(T)$, if $B$ is a nice multi curve chains 	such that $\hat{B}=0$.

Suppose we have the following data:
\begin{itemize}
\item A tree $T_0$ without edges,
\item For each $v \in V(T_0)$ a one dimensional simplicial cycle $\gamma_v$ such that $[\gamma_v]= \partial \beta_v$ and do not intersect pairwise.
\end{itemize}
For each tree $T$ with the same vertices of $T_0$ define 
\begin{equation} \label{nice-chains}
W_{T_0, \{\gamma_v \}_{v \in V(T_0)}}(T) = \prod_{e \in H(T)} \gamma_{\pi(e)}.
\end{equation}

\begin{lemma}
Formula (\ref{nice-chains}) defines a nice multi disk $0$-chain. Each nice multi disk $0$-chain is a linear combination of chains of this type.
\end{lemma}
\begin{proof}
Suppose that $W$ is nice multi disk $0$-cycle. From the definition it follows that  $W(T)$ belongs to $\mathcal{C}_d(T, \{E_0 \})$ for each $E_0$. This means that $W(T)$ has to be forgetful compatible with respect all the edges and transverse to all the big diagonals. The lemma follows.
\end{proof}

An analogous statement holds for nice multi-disks chains of higher dimension. To define rational invariants, in this paper we are interested only to dimension $0$ and $1$. To define a nice one-chain we use a one parameter family of data before, that is for every $v \in V(T_0)$ we fix a one family of curves  $\{ t \rightarrow \gamma_{v,t}  \}_{t \in [a,b]}$. The space of nice one dimensional chains is generated by isotopy of (\ref{nice-chains}):
\begin{equation} \label{nice-chains1}
 B(T) = \{ t \rightarrow   \prod_{e \in H(T)} \gamma_{\pi(e),t}  \}_{t \in [a,b]}.
\end{equation}

Consider an isotopy such that $\gamma_{v_1,t}$ crosses $\gamma_{v_2,t}$ at some $t_0$. It follows directly from the definition of $\hat{\partial}$ that:
\begin{equation} \label{relation}
\hat{\partial} B =  W_{T_1, \{ \{\gamma_{v,t_0}  \}_{v \in V(T_0)\setminus \{ v_1,v_2 \} }, \gamma_{v_1,t_0} +\gamma_{v_2,t_0} \} } .
\end{equation}
where $T_1$ is the tree with no edges that we get from $T_0$ replacing the vertices $v_1$ and $v_2$ with one only vertex in homology class $\beta_{v_1} +\beta_{v_2}$.

As in \cite{OGW3}, we can prove that the following
\begin{proposition} \label{niceiso}
The natural map 
\begin{equation} \label{nicemap}
NMCH_0(\beta,S,Z,Q) \rightarrow MCH_0(\beta,S,Z,Q)
\end{equation}
is an isomorphism.
\end{proposition} 
\begin{proof}
The proof is similar to the one of \cite{OGW3}. We first prove that (\ref{nicemap}) is surjective. 
Let $W \in \hat{\mathcal{C}}_{0}(\beta,S,Z;Q)$ be a multi-disk $0$-cycle, we need to show that there exists $B \in  \hat{\mathcal{C}}_{1}(\beta,S,Z,Q)$ such that $W'=W + (\hat{\partial } B)$ is a nice multi disk chain. We prove this by a double induction.

As in \cite{OGW3}, let $\overline{k}  \in \Z_{\geq 0}$ be a positive integer number, and assume that  
\begin{equation} \label{nice-induction0}
W_{k,l} = 0 \text{ if } k >\overline{k}  \text{ and }  l>0 .
\end{equation}
We want to prove that there exists $B$ such that $Z'_{\overline{k} ,l}=0$ if  $l > 0$. We will define $B$ such that $B_{k,l}=0$ if $k \neq \overline{k} $, in particular (\ref{nice-induction0}) still holds for $W'$ instead of $W$.

We first define $B_{\overline{k},0}$.  For $T \in \mathcal{T}_{\overline{k}}(\beta,S,Z,Q)$,  applying Definition \ref{forget} with $E_0 = E(T)$,  there exists curves $ \{   \gamma_{v}  \}_{v \in V(T)} $ such that 
$W_{\overline{k},0}(T, \{ E(T) \} )   = \prod_{h \in H(T)} \gamma_{\pi(h)}$.
Moreover the curves $ \{   \gamma_{v}  \}_{v \in V(T)} $ depends only by the tree $T'$ defined as in Definition \ref{forget}, that is a tree with $\overline{k}+1$ vertices and no edges.
For each $v \in V(T')$ fix a small perturbation $\tilde{\gamma}_v$ of $\gamma_v$, such that the support of the curves do not intersects pairwise.   
Define 
$$W'_{\overline{k},0}(T) =  \prod_{h \in H(T)} \tilde{\gamma}_{\pi(h)} . $$ 
The current $W'_{\overline{k},0}(T)$ is a small pertubation of $W_{\overline{k},0}(T, \{ E(T) \} )$ and, since we are working with trees,  is transverse along every big diagonals (this step would not be true in the case of graphs considered in \cite{OGW3}). From $ \hat{\partial}  W =0$ we know that  $ \partial W_{\overline{k},1}(T,\{ E_0, E(T) \})=  W_{\overline{k},0}(T,\{ E_0 \})- W_{\overline{k},0}(T,\{ E(T) \})$, so there exists $B_{\overline{k},0}(\{ E_0 \})$ small perturbation of $ W_{\overline{k},1}(T,\{E_0, E(T) \})$, transverse to all the big diagonals associated to the edges in $E(T) \setminus E_0$, such that the identity $\partial B_{\overline{k},0}(E_0)=  W_{\overline{k},0}(T,E_0) -  W'_{\overline{k},0}(T)$ holds.

Now we proceed by induction on $l$, as in \cite{OGW3}. Assume that we have defined $B_{\overline{k} ,l}$ for every $l < \overline{l} $ with the following proprieties: 
\begin{itemize}
\item  for every $(T, \{  E_0,E_1,...,E_l \}) \in \mathcal{T}_{\overline{k},l}$,
$B_{\overline{k},l}(T, \{  E_0,E_1,...,E_l \})$ is a small perturbation of $(-1)^{l+1} W_{\overline{k},l+1}(T, \{  E_0,E_1,...,E_l, E(T) \})$,
\item   for every $(T, \{  E_0,E_1,...,E_l \}) \in \mathcal{T}_{\overline{k},l}$,
$B_{\overline{k},l}(T, \{  E_0,E_1,...,E_l \})$ is transverse to the big diagonals associated to the edges $E(T ) \setminus E_l$,

\item The following identity holds for $0< l < \overline{l}$
\begin{equation} \label{nice-induction}
Z_{\overline{k},l} + \partial B_{\overline{k},l}  + \tilde{\partial} B_{\overline{k},l-1} =0 .
\end{equation}
\end{itemize}
The proof goes throughout as in \cite{OGW3}. 

Also the injectivity of  (\ref{nicemap}) can be proved as in \cite{OGW3}. 
\end{proof}

\subsection{Cup Product}

We now introduce two definitions of product in Multi-Disk-Homology with central charges.

For $\beta_1, \beta_2 \in \Gamma$, define 
$$\mathfrak{T}_{k, l}(\beta_1,\beta_2;S,Z,Q) = \sqcup_{k_1+k_2=k} \mathfrak{T}_{k_1, l}(\beta_1;S,Z,Q) \times \mathfrak{T}_{k_2, l}(\beta_2;S,Z,Q).$$
We define  $MCH(\beta_1,\beta_2;S,Z,Q)$ using  $\mathfrak{T}(\beta_1,\beta_2;S,Z,Q)$, in a way completely parallel to what we have done above for $MCH(\beta;S,Z,Q)$.
The union of trees
$$ ((T_1,\{ E_{i,1}  \}_{0 \leq i \leq l}) , (T_2, \{ E_{i,2}  \}_{0 \leq i \leq l})) \rightarrow  (T_1 \sqcup T_2,\{ E_{i,1} \sqcup E_{i,2} \}_{0 \leq i \leq l})$$
defines a map 
\begin{equation} \label{uniontree1}
 \mathfrak{T}(\beta_1,\beta_2,S,Z,Q) \rightarrow  \mathfrak{T}(\beta_1 + \beta_2,S,Z,Q).
\end{equation}

The map (\ref{uniontree1}) induces the homomorphism
\begin{equation} \label{restriction}
 \mathcal{F}_{\beta_1,\beta_2} :  MDH_0(\beta,S,Z,Q) \rightarrow  MDH_0(\beta_1,\beta_2,S,Z,Q).
\end{equation}

We now define the cup product in $MDH$:
$$ \cup : MDH_0(\beta_1,S,Z,Q)  \times MDH_0(\beta_2,S,Z,Q) \rightarrow  MDH_0(\beta_1,  \beta_2 ,S,Z,Q) .$$
For $W_1 \in MDH_0(\beta_1,S,Z,Q)$ and $W_2 \in MDH_0(\beta_2,S,Z,Q)$, $ W_1 \cup W_2 \in  MDH_0(\beta_1,  \beta_2 ,S,Z,Q)$ is given by the formula
\begin{equation} \label{cup-formula}
(W_1\cup W_2) ((T_1,\{ E_{i,1}  \}_{0 \leq i \leq l}) , (T_2, \{ E_{i,2}  \}_{0 \leq i \leq l}))= \sum_{0 \leq r \leq l}  W_1(T_1,  \{ E_{i,1}  \}_{0 \leq i \leq r}) \times  W_2(T_2, \{ E_{i,2}  \}_{r \leq i \leq l})   . 
\end{equation}
It is easy to check that 
$$  \hat{\partial}  (W_1\cup W_2) =  \hat{\partial} (W_1) \cup W_2  + W_1 \cup  \hat{\partial} (W_2)  $$
and therefore $\cup$ descends in homology.

Suppose now we have two disjoint strictly convex sectors $S_1 , S_2 \subset \R^2$ such that $S= S_1 \sqcup S_2 $ in clock-wise order. Define the set  
$$\mathfrak{T}_{k, l}(\beta, S_1,S_2,Z,Q) = \sqcup_{\beta_1 + \beta_2 = \beta} \sqcup_{ k_1+k_2=k} \mathfrak{T}_{k_1, l}(\beta_1,S_1,Z,Q) \times \mathfrak{T}_{k_2, l}(\beta_2, S_2,Z,Q)$$
and use it to define $MCH(\beta,S_1,S_2,Z,Q)$ as above.
Consider the map
\begin{equation} \label{uniontree2}
\mathfrak{T}_{k, l}(\beta, S_1,S_2,Z,Q) \rightarrow  \mathfrak{T}(\beta, S,Z,Q )
\end{equation}
as in   (\ref{uniontree1})
inducted by the union of trees. The associated homomorphism is denoted as
\begin{equation} \label{restriction2}
 \mathcal{F}_{S_1,S_2} :  MDH_0(\beta,S,Z,Q) \rightarrow  MDH_0(\beta, S_1,S_2,Z,Q).
\end{equation}

The second version of  cup product is:
$$ \cup : MDH_0(\beta_1,S_1,Z,Q)  \times MDH_0(\beta_2,S_2,Z,Q) \rightarrow  MDH_0(\beta_1 +\beta_2, S_1,S_2,Z,Q ) $$
where for $W_1 \in MDH_0(\beta_1,S_1,Z,Q)$ and $W_2 \in MDH_0(\beta_2,S_2,Z,Q)$, the definition of $ W_1 \cup W_2 \in MDH_0(\beta, S_1,S_2,Z,Q)$ is given by the same formula  as in (\ref{cup-formula}).

 



\section{Rational Invariants and Wall Crossing}

Endow $\Gamma$ with the skew-symmetric pairing $$\langle \bullet , \bullet \rangle : \Gamma \times \Gamma \rightarrow \Z$$ given by 
$$ \langle \beta_1 , \beta_2 \rangle = \langle(\partial \beta_1) , (\partial \beta_2)   \rangle $$
where in the right side we have used the intersection pairing on $H_1(L, \Z)$.

To the abelian group $\Gamma$ we associate the Lie algebra $\mathfrak{g}_{\Gamma(S,Z,Q)} = \oplus_{\beta \in \Gamma(S,Z,Q)} \Q e_{\beta}$ with Lie bracket
$$ [ e_{\beta_1} , e_{\beta_2} ] =  \langle \beta_1 , \beta_2 \rangle e_{\beta_1 + \beta_2} . $$
This is the same algebra introduced in \cite{KS} except for the sign. 

Let $U(\mathfrak{g}_{\Gamma(S,Z,Q)})$ be the enveloping algebra of the Lie algebra $\mathfrak{g}_{\Gamma(S,Z,Q)}$. We denote $U(\mathfrak{g}_{\Gamma(S,Z,QC})_{\beta}$ the vector subspace of $U(\mathfrak{g}_{\Gamma(S,Z,Q)})$ generated by the elements $e_{\beta_1} \otimes e_{\beta_2} \otimes .... \otimes e_{\beta_r}$ with $\sum_i \beta_i = \beta$.

\begin{proposition} \label{isomorphism}
There is a natural isomorphism 
$$ MDH_0 (\beta;S.Z,Q) \simeq  U(\mathfrak{g}_{\Gamma(S,Z,Q)})_{\beta}  $$
\end{proposition}
\begin{proof}
First note that in formula  (\ref{nice-chains}) we can assume that the curves $\gamma_v$ have support on $\Sigma \times \{ \theta_v \}$ for some $\theta_v \in I_S$. Its image on $MDH_0$ depends only on $\beta_v$ and the ordering of $  \{ \theta_v \}_{ v \in V(T_0)}$. This subset of multi disks chains still generate $MDH_0$.  Moreover the set of relation that they satisfy in $MDH_0$ is generated by isotopies that exchange the order of  $  \{ t_v \}_{ v \in V(T_0)}$.

Define 
\begin{equation} \label{main-isomorphism}
W_{T_0, \{\gamma_v \}_{v \in V(T_0)}}  \longmapsto e_{\beta_{v_0}} \otimes e_{\beta_{v_1}} \otimes .... \otimes e_{\beta_{v_r}}  
\end{equation}
where $ \theta_{v_i} <  \theta_{v_j} $ if $i < j$. We want to prove that (\ref{main-isomorphism}) induces a map   $MDH_0 (\beta;S.Z,Q) \rightarrow  U(\mathfrak{g}_{\Gamma(S,Z,Q)})_{\beta} $ that is an isomorphism. In order to do this we need to check that the relation that the elements in satisfy in the two sides matches.
 
Suppose that we exchange the order of  $t_j$ and $t_{j+1}$. In the isotopy the number of cross (counted with sign) of the curves $\gamma_{v_j}$ and  $\gamma_{v_{j+1}}$ is given by $ \langle \gamma_j, \gamma_{j+1} \rangle$. Applying (\ref{relation}) for each cross we get that the jumping  of the left side of (\ref{main-isomorphism}) in $U(\mathfrak{g}_{\Gamma(S,Z,Q)})_{\beta}$ is mapped in  $\langle \gamma_j, \gamma_{j+1} \rangle  e_{\beta_{v_0}} \otimes... \otimes e_{\beta_{v_{j-1} } }\otimes e_{\beta_{v_j}  + \beta_{v_{j+1}}} \otimes e_{\beta_{v_{j+2}} }\otimes .... \otimes e_{\beta_{v_r}}$. This means that the relations on the two sides matches.

\end{proof}

\subsection{Rational Invariants}

Assume now that $Z$ is outside a wall of the first type. Recall from \cite{KS} that by definition, $Z$ is outside of \emph{wall of the first type}  if  
$$ Z(\beta_1) || Z(\beta_2)  \Longrightarrow   \beta_1 || \beta_2 ,$$ 
for each $ \beta_1 , \beta_2 \in \Gamma$. We now prove that in this case we can associate to each $W \in MCH_0(\beta;S,Z,Q)$ a rational number. 

To define the rational invariant as in \cite{OGW3}, we need a notion of linking number between pairs of curves on $L= \Sigma \times I_S$. In general a linking number  is an integer number $Link  ( \gamma_1 , \gamma_2 ) $ defined for each pair of not intersecting curves $\gamma_1,\gamma_2$ on  $L$ such that 
\begin{itemize}
\item $Link  ( \gamma_1 , \gamma_2 ) $ is invariant by deformations of $\gamma_1, \gamma_2$,
\item $Link  ( \gamma_1 , \gamma_2 ) $  jumps by $\pm 1$ if $\gamma_1$ and $\gamma_2$ cross transversely according to the sign of the cross.
\end{itemize}
In the context of this paper we need a slight refinement of this notion. Instead of considering curves on $L$ we need to consider pairs $(\gamma, \beta)$ where $\gamma$ is a curve on $L$ and $\beta \in \Gamma$ with $[\gamma]= \partial \beta$. We are then interested to define for each pair $ (\gamma_1, \beta_1) , (\gamma_2, \beta_2) $
\begin{equation} \label{link}
Link  ( (\gamma_1, \beta_1) , (\gamma_2, \beta_2) ) \in \Z.
\end{equation}
as an integer number that jumps by $\pm 1$ if $\gamma_1$ and $\gamma_2$ cross transversely according to the sign of the cross.
\begin{lemma}
There exists a unique definition of (\ref{link}) such that
$$\text{Arg}  (Z(\beta_1)) < \text{Arg}(Z(\beta_2)), \theta_1 < \theta_2   \Longrightarrow   Link  ( (\gamma_1, \beta_1) , (\gamma_2, \beta_2) )=0$$
if the support of $\gamma_i$ is in  $\Sigma \times \{ \theta_i \}$.
\end{lemma}
The lemma follows easily from the fact that the intersection pairing between $H_2(L, \Z)$ and $H_1(L, \Z)$ vanishes (if $\Sigma$ is non-compact we actually have $H_2(L)=0$). An analogous fact was used in \cite{frame}.

We can now apply the definion of \cite{OGW3} to get the multi-link homorphim
\begin{equation} \label{rational}
MultiLink: MDH_0(\beta,S,Z,Q) \rightarrow H_0(point) = \Q.
\end{equation} 
On the nice chains this is defined as 
\begin{equation} \label{multi-link}
MultiLink(W_{T_0, \{\gamma_v \}_{v \in V(T_0)}}(T)) = \frac{1}{k!}  \prod_{ \{h_1, h_2\}  \in E(T)} Link((\gamma_{\pi(h_1)}, \beta_{\pi(h_1)}) , (\gamma_{\pi(h_2)},  \beta_{\pi(h_2)} )) .
\end{equation}

\subsection{String Stability Structure}

We now introduce the analogous of the Stability Structures defined in Section $2$ of \cite{KS}.

Let $Spin(\Sigma)$ the set of spin structures of $\Sigma$. Recall that there is an obvious action of $H^1(\Sigma, \Z_2)$ on $Spin(\Sigma)$.
\begin{definition}
A string stability structure consist of
\begin{itemize}
\item A quadratic form $Q$ on $\Gamma_{\R}$,
\item An homomorphism of abelian groups  $Z: \Gamma \rightarrow \R^2$ such that $Q |_{\text{Ker}(Z)}< 0$
\item For each strictly convex sector $S$ on $\R^2$, $\beta \in \Gamma$, $\sigma \in Spin(\Sigma)$ an element $W^{\sigma}_{\beta,S} \in MDH(\beta;S,Z,Q)$
\end{itemize} 
satisfying the following proprieties
\begin{itemize}
\item For each $\epsilon \in H^1(\Sigma, \Z_2)$
$$ W^{\epsilon \sigma}_{\beta,S} = \epsilon (\partial \beta) W^{ \sigma}_{\beta,S} $$
\item For every $\beta_1, \beta_2,...,\beta_r \in \Gamma$ with $\beta_1+\beta_2+...+\beta_r= \beta$ 
\begin{equation} \label{factorization1}
 \mathcal{F}_{\beta_1,\beta_2,...,\beta_r}(W^{\sigma}_{\beta,S}) = [W^{\sigma}_{\beta_1,S} \cup W^{\sigma}_{\beta_2,S} \cup ... \cup W^{\sigma}_{\beta_r,S}] / \text{Aut} \{ \beta_1,\beta_2,...,\beta_r  \} .
\end{equation}
\item For two disjoint sectors $S_1 , S_2 \subset\R^2$ such that $S= S_1 \sqcup S_2 $ in clock-wise order
\begin{equation} \label{factorization2} 
\mathcal{F}_{S_1,S_2}(W^{\sigma}_{\beta,S}) =  \sum_{\beta_1 + \beta_2 = \beta} [ W^{\sigma}_{\beta_1,S_1}\cup W^{\sigma}_{\beta_2,S_2}] . 
\end{equation}
\end{itemize}
\end{definition}

For each $\beta,S, \sigma$ define
$$ a^{\sigma}_{\beta,S} = MultiLink(W^{\sigma}_{\beta,S}) $$
using (\ref{rational}).


Denote by $A^{\sigma}_{\beta,S}$ the image of $W^{\sigma}_{\beta,S}$ in $U(\mathfrak{g}_{S,Z,Q})_{\beta}$ given by Proposition \ref{isomorphism}. 
Define 
$$ A^{\sigma}_S =  \sum_{\beta} A^{\sigma}_{\beta,S}   \in  \prod_{\beta \in \Gamma}  U(\mathfrak{g}_{S,Z,Q})_{\beta} .$$
The right side can be a infinite sum, so does not belong to $U(\mathfrak{g}_{\Gamma(S,Z,Q)}) $. We consider it as a formal sum on the infinite direct product $ \prod_{\beta \in \Gamma}  U(\mathfrak{g}_{S,Z,Q})_{\beta}$. Observe that this is still an algebra since any element of $\Gamma(S,Z,Q)$ can be written as the sum of other elements of  $\Gamma(S,Z,Q)$ in a finite number of way. 

\begin{proposition}
If $Z$ is outside of a wall of first type, $A^{\sigma}_S$ is given by the following formula 
\begin{equation} \label{WCF}
A^{\sigma}_S = \prod_{l \subset S }^{\rightarrow} \text{exp} ( \sum_{Z(\beta) \in l}  a^{\sigma}_{\beta,S} e_{\beta}  ) .
\end{equation} 
\end{proposition}
Here we are using the same notation of \cite{KS}. In (\ref{WCF}) the product is taken in the clockwise order over the set of all rays $l$ which belong to the sector $S$. 
Only a finite numbers of rays contribute to the component on $ U(\mathfrak{g}_{\Gamma(S,Z,Q)})_{\beta}$.
\begin{proof}
As in Proposition \ref{isomorphism}), we can consider generators of $MDH_0$ given by formula(\ref{nice-chains}) such that the curves $\gamma_v$ have support on $\Sigma \times \{ \theta_v \}$ for some $\theta_v \in I_S$.  Moreover we can assume that $ \{Z(\beta_1), Z(\beta_2),...,Z(\beta_r)  \} $ are in clock-wise order.
In this case all the linking numbers involved in the definition (\ref{rational}) are zero, so the coefficient of $e_{\beta}$ in $A^{\sigma}_S$ is given by $a^{\sigma}_{\beta,S}$ by definition. 

The factorization propriety (\ref{factorization1}) implies that the coefficient of $e_{\beta_1} \otimes e_{\beta_2} \otimes .... \otimes e_{\beta_r}$ in $A^{\sigma}_S$ is given by 
$$\frac{1}{|\text{Aut} \{ \beta_1,\beta_2,...,\beta_r  \}|}\prod_i a^{\sigma}_{\beta_iS} .$$ 
This match the coefficient on the right hand side of (\ref{WCF}). 
\end{proof}

From (\ref{factorization2}) it follows  directly that 
\begin{equation} \label{KS-factotization}
 A^{\sigma}_S =  A^{\sigma}_{S_1} A^{\sigma}_{S_2}
\end{equation}
for any two disjoint sectors $S_1 , S_2 \subset\R^2$ such that $S= S_1 \sqcup S_2 $ in clock-wise order. In particular (\ref{KS-factotization}) and (\ref{WCF}) imply that the rational numbers $a^{\sigma}_{\beta,S}$ do not depend on the sector $S$. We denote this rational number as $a^{\sigma}_{\beta}$.

\subsection{Quadratic Refinements}
Recall that a quadratic refinement of the intersection pairing of $H_1(\Sigma,\Z)$ is a map $\sigma: H_1(\Sigma, \Z) \rightarrow \Z_2$ such that 
$$ \sigma (\gamma_1) \sigma (\gamma_2) = (-1)^{ \langle \gamma_1 , \gamma_2 \rangle  }  \sigma (\gamma_1+ \gamma_2)  .$$
There is transitive action of $H^1(\Sigma, \Z_2)$ on the set of quadratic refinement given by
$$ \sigma(\gamma) \rightarrow (-1)^{\epsilon(\gamma)} \sigma(\gamma)$$
for $\epsilon \in H^1(\Sigma, \Z_2)$.

It is a classical result of Atiyah that there exists a one-to-one correspondence between spin structures on $\Sigma$ and quadratic refinement
$$ \text{ spin structures on }   \Sigma \longleftrightarrow \text{ quadratic refinement of  }   (-1)^{ \langle \cdot , \cdot \rangle}  $$
that is compatible with the action of $H^1(\Sigma, \Z_2)$ on both the sides.

Define the algebra
$\hat{\mathfrak{g}}_{\Gamma(S,Z,Q)} = \oplus_{\beta \in \Gamma(S,Z,Q)} \Q \hat{e}_{\beta}$ with Lie bracket
$$ [ \hat{e}_{\beta_1} , \hat{e}_{\beta_2} ] =  (-1)^{ \langle \beta_1 , \beta_2 \rangle}   \langle \beta_1 , \beta_2 \rangle \hat{e}_{\beta_1 + \beta_2} . $$
This is the algebra introduced in \cite{KS}.

Each $\sigma \in Spin(\Sigma)$ defines an isomorphism of algebras
\begin{equation} \label{sigma-iso}
 \mathfrak{g}_{\Gamma(S,Z,Q)} \rightarrow \hat{\mathfrak{g}}_{\Gamma(S,Z,Q)} 
\end{equation}
given by $ {e}_{\beta} \rightarrow \sigma(\partial \beta) \hat{e}_{\beta}  $. 

The rational number
$$ a_{\beta} = \sigma(\partial \beta)  a^{\sigma}_{\beta} $$
does not depend on the spin structure $\sigma$. 
The image  $\hat{A}_S$  of $ A^{\sigma}_S $ by the isomorphism (\ref{sigma-iso}) does not depend on $\sigma$ and is given by
\begin{equation} \label{WCF2}
\hat{A}_S = \prod_{l \subset S }^{\rightarrow} \text{exp} ( \sum_{Z(\beta) \in l} a_{\beta} \hat{e}_{\beta}  ) .
\end{equation}

\subsection{Variation of String Stability Structure}

Let $\beta \in \Gamma$ and $S \subset \R^2$ a strictly convex sector. We say that a central charge $Z$ is on a $(S, \beta)$ wall of second type if exist $\beta_1, \beta_2 \in \Gamma(S,Z,Q)$ with  $Z(\beta_1) \in \partial S$, such that $\beta = \beta_1 + \beta_2$.

A variation of String Stability Structures is one parameter family of string stability structures $(Z_t , \{ W^{\sigma}_{\beta,S,t} \}_{\beta,S }  )_{0 \leq t \leq 1} $ such that $W^{\sigma}_{\beta,S,t }$ does not depend on $t \in [0,1]$ if $Z_t$ is not $(S, \beta)$ wall of second type for each $t \in [0,1]$.


\begin{proposition}
The component in $ U(\mathfrak{g}_{\Gamma(S,Z,Q)})_{\beta}$ of the element (\ref{WCF2}) is invariant  in a variation of String Stability Structure such that $(Z_t  )_{0 \leq t \leq 1}$ does not hit a $\beta$-wall of second type.
\end{proposition}
\begin{proof}
$ A^{\sigma}_{\beta,S,t}$ is invariant since $ W^{\sigma}_{\beta,S,t}$ is invariant. 
\end{proof}



\section{Geometric Realization}

Let $(X,\Omega)$ be a complex symplectic manifold of complex dimensions two. Let $\Sigma$ be a complex-lagrangian submanifold of $X$. The pair $(X, \Sigma)$ is not necessarily compact. In the non compact case we assume suitable convexity propriety at infinity which ensure compactness of the moduli space of curves. In this section we associate to the pair  $(X,\Sigma)$ a string stability structure.

Fix a quasi-hyperkahler structure $(\omega_1,J_1),(\omega_2,J_2),  (\omega_3,J_3)$ with $\Omega= \omega_1 +i \omega_2$. The differential form $\omega_3$ is not necessary closed.

To the pair $(X,\Sigma)$  we associate:
\begin{itemize}
\item The surface $\Sigma$ as above 
\item Topological charges $\Gamma = H_2(X, \Sigma, \Z)$. 
\item $\partial : \Gamma \rightarrow H_1( \Sigma, \Z) $
is the usual boundary map in homology.
\end{itemize}

Observe that the data above do are invariant (in a obvious sense) by the isotopies of $\Sigma $ inside $X$. To define an element of $MDH$ the central charge is given by
$$ Z(\beta) = \int_{\beta} \Omega . $$

For $\theta \in \R / (2 \pi \Z) $, let $J_{\theta}= \text{cos}(\theta) J_1  + \text{sen}(\theta) J_2  $ be the almost-complex structure compatible with the symplectic structure $\text{Re} (e^{- i \theta} \Omega)$.
For $\beta \in H_2(X, \Sigma, \Z)$, let $\overline{\mathcal{M}}_k (\beta, J_{\theta})$ be the moduli space of $J_{\theta}$-psudoholomphic disks in the relative homology class $\beta$, with boundary mapped on $\Sigma$ and $k$ boundary punctures. It is a standard fact (easy to prove) that   
$$\overline{\mathcal{M}}_k (\beta, J_{\theta}) \neq \emptyset \Rightarrow \text{Arg}(Z(\gamma)) = \theta   .$$

The virtual dimension of $\overline{\mathcal{M}}_k (\beta, J_{\theta})$ is $-1$. For each strictly convex sector $S \subset \R^2$, we consider the augmented moduli space  
\begin{equation} \label{augmented}
  \overline{\mathcal{M}}_k(\beta) =  \bigsqcup_{ \theta | \R \cdot e^{ i \theta } \subset S   }   \overline{\mathcal{M}}_k (\beta, J_{\theta}) \times \{ \theta \}
\end{equation}
that has virtual dimension zero.
The evaluation map 
$$ \text{ev}_i : \overline{\mathcal{M}}_k(\beta)  \rightarrow L = \Sigma \times I_S $$
is given by $ \text{ev}_i = ( \text{ev}_{\theta,i}, \theta  )   $,
where $\text{ev}_{\theta,i}: \overline{\mathcal{M}}_k (\beta, J_{\theta})  \rightarrow \Sigma $ is the usual evaluation map on the boundary marked point labeled $i$.


Fix a norm $ \parallel \dot \parallel  $ on $\Gamma_{\R}$.  It is a standard fact (easy to prove)  that there exists a constant $C$ such that 
$$  \overline{\mathcal{M}}_k(\beta)  \neq \emptyset \Rightarrow \parallel \beta \parallel  \leq C |Z(\beta)| .$$
The quadratic form $Q$ is defined as
$$   Q(\beta) = -  \parallel \beta \parallel ^2  + C^2 |Z(\beta)|^2   .$$

\begin{proposition}
To the pair $(X, \Sigma)$ it is associated a natural string stability structure.
\end{proposition}
\begin{proof}
As in \cite{OGW3}, use the moduli spaces (\ref{augmented}) to define moduli spaces $\overline{\mathcal{M}}_{(T,m)}$ for every $(T,m) \in \mathfrak{T}_{k,l}(\beta,S,Z,Q)$ and define $W^{\sigma}_{\beta,S} \in MDH_0(\beta.S,Z.Q)$ that is well defined in homology. The argument is completely analogous to the one of \cite{OGW3}.   We want to prove that this data define a string stability structure. For this we need to to prove that the propriety (\ref{factorization1}) holds (the proof of (\ref{factorization2}) is similar).

Since the spin structure is fixed, we shall omit $\sigma$ in the notation below. 
We need to prove that the image of $W_{\beta,S}$ under (\ref{restriction}) is given by 
$$ [ W_{\beta_1,S} \cup W_{\beta_2,S} \cup ... \cup  W_{\beta_r,S} ] / \text{Aut} \{ \beta_1,\beta_2,...,\beta_r  \}$$


Consider for simplicity the case $r=2$. 
First define $W_{(\beta_1,\beta_2,S)} \in MDH_0((\beta_1,\beta_2,S,Z,Q))$ using the moduli space associated to $\mathfrak{T}(\beta_1,\beta_2,S,Z,Q)$. From the perturbations of the moduli spaces $\overline{\mathcal{M}}_{(T,m)}$, we now define two perturbations of the moduli spaces  $ \overline{\mathcal{M}}_{ (T_1,\{ E_{i,1}  \}) , (T_2, \{ E_{i,2}  \})} $.

Observe first that we can realize the $l$-simplex $\Delta_l= [0,1,...,l]$ as product of simplexes using the homeomerphism
$$ \Delta_l = (\bigsqcup_{0 \leq v \leq l} [0,1,...,v] \times [v,v+1 ...,l] )/ \sim $$
where the face $\partial_{v+1} [0,1,...,v, v+1] \times [v+1, ...,l]$ of  $[0,1,...,v, v+1] \times [v+1, ...,l]$ is identified with the face $ [0,1,..., v] \times \partial_{v} [v, ...,l]$ of  $[0,1,...,v] \times [v, ...,l]$. 
This decomposition induces an isomorphism of spaces 
\begin{equation} \label{moduli-factorization}
 \overline{\mathcal{M}}_{ (T_1,\{ E_{i,1}  \}) , (T_2, \{ E_{i,2}  \})} = \bigsqcup_{0 \leq v \leq l}  \overline{\mathcal{M}}_{(T_1,  \{ E_{i,1}  \}_{0 \leq i \leq v})} \times \overline{\mathcal{M}}_{(T_2,  \{ E_{i,2}  \}_{v \leq i \leq l})}.
\end{equation}
Define on $ \overline{\mathcal{M}}_{ (T_1,\{ E_{i,1}  \}) , (T_2, \{ E_{i,2}  \})} $ the Kuranishi structure and perturbation induced by the right side of (\ref{moduli-factorization}):
$$\mathfrak{s}_{ (T_1,\{ E_{i,1}  \}) , (T_2, \{ E_{i,2}  \})}=  \bigsqcup_v (\mathfrak{s}_{(T_1,  \{ E_{i,1}  \}_{0 \leq i \leq v})} \times \mathfrak{s}_{(T_2,  \{ E_{i,2}  \}_{v \leq i \leq l})} ) .$$ 
The associated virtual fundamental class is by definition $ W_{\beta_1,S} \cup W_{\beta_2,S} \in MDH_0(\beta_1,\beta_2,S,Z,Q)$.


The second perturbation of  $ \overline{\mathcal{M}}_{ (T_1,\{ E_{i,1}  \}) , (T_2, \{ E_{i,2}  \})} $ is defined using the pull-back of (\ref{uniontree1})
$$\mathfrak{s}_{ (T_1,\{ E_{i,1}  \}) , (T_2, \{ E_{i,2}  \})}  = \mathfrak{s}_{ (T_1 \sqcup T_2,\{ E_{i,1}  \sqcup E_{i,2}  \})} $$ 
The associated virtual fundamental class is the image of $W_{\beta,S}$ on  $MDH_0(\beta_1,\beta_2,S,Z,Q)$.

We know that the representative on $MDH_0(\beta_1,\beta_2,S,Z,Q)$ does not depend on the perturbation. Propriety (\ref{factorization1}) then follows. 

\end{proof}

\begin{theorem}
To an isopoty of the complex lagrangian $\Sigma$ it is associated a variation of string stability structures. In particular the element (\ref{WCF2}) is invariant when we cross a wall of first type.
\end{theorem}
\begin{proof}
 $W$ is invariant as element of $ MDH$ since there are not holomorphic disks of augmented dimension $-1$.
\end{proof}

\end{document}